\documentclass[a4paper,12pt]{article}
\usepackage[top=1in,left=1in,right=1in,bottom=1.5in]{geometry}

\usepackage[T1]{fontenc}
\usepackage{eucal,mathrsfs,dsfont,stmaryrd}
\usepackage{amsmath,amsfonts,amssymb,amsthm}
\usepackage{mathtools}

\usepackage{pstricks,pst-text,pst-plot,graphicx}
\usepackage{paralist,enumerate}

\theoremstyle{plain}
\newtheorem{theorem}{Theorem}[section]
\newtheorem{corollary}[theorem]{Corollary}
\newtheorem{lemma}[theorem]{Lemma}
\newtheorem{proposition}[theorem]{Proposition}

\theoremstyle{definition}
\newtheorem{definition}[theorem]{Definition}

\theoremstyle{remark}
\newtheorem{remark}[theorem]{Remark}
\newtheorem{example}[theorem]{Example}

\renewcommand{\leq}{\leqslant}
\renewcommand{\geq}{\geqslant}
\renewcommand{\Re}{\ensuremath{\operatorname{Re}}}
\renewcommand{\Im}{\ensuremath{\operatorname{Im}}}

\newcommand{\bbc}{\mathds{C}}
\newcommand{\bbr}{\mathds{R}}
\newcommand{\rd}{{\mathds{R}^d}}
\newcommand{\bbn}{\mathds{N}}
\newcommand{\Pp}{\mathds{P}}
\newcommand{\Ee}{\mathds{E}}
\newcommand{\I}{\mathds{1}}

\newcommand{\cb}{\mathcal{B}}
\newcommand{\cf}{\mathcal{F}}
\newcommand{\cm}{\mathcal{M}}

\newcommand{\abs}[1]{\left| #1 \right|}
\newcommand{\nnorm}[1]{\left\| #1 \right\|}

\begin{document}

\title{\bfseries The Symbol Associated with the Solution of a Stochastic Differential Equation}

\author{%
    \textsc{Ren\'e L. Schilling}%
    \thanks{Institut f\"ur Mathematische Stochastik,
              Technische Universit\"at Dresden,
              D-01062 Dresden, Germany,
              \texttt{rene.schilling@tu-dresden.de}}
    \textrm{\ \ and\ \ }
    \stepcounter{footnote}\stepcounter{footnote}\stepcounter{footnote}
    \stepcounter{footnote}\stepcounter{footnote}%
    \textsc{Alexander Schnurr}%
    \thanks{Lehrstuhl IV, Fakult\"at f\"ur Mathematik, Technische Universit\"at Dortmund,
              D-44227 Dortmund, Germany,
              \texttt{alexander.schnurr@math.tu-dortmund.de}}
    }

\date{}

\maketitle
\begin{abstract}\small
Let $(Z_t)_{t\geq 0}$ be an $\bbr^n$-valued L\'evy process. We
consider stochastic differential equations of the form
\begin{align*}
  dX_t^x&=\Phi(X_{t-}^x) \,dZ_t \\
  X_0^x&=x,\quad x\in\bbr^d,
\end{align*}
where $\Phi: \bbr^d \to \bbr^{d \times n}$ is Lipschitz continuous. We show that the infinitesimal generator of the solution process  $(X^x_t)_{t\geq 0}$ is a pseudo-differential operator whose symbol $p:\bbr^d\times\bbr^d\to \bbc$ can be calculated by
\begin{gather*}
    p(x,\xi):=- \lim_{t\downarrow 0}\Ee^x \left(\frac{e^{i(X^\sigma_t-x)^\top\xi}-1}{t}\right).
\end{gather*}
For a large class of Feller processes many properties of the sample paths can be derived by analysing the symbol. It turns out that the process $(X^x_t)_{t\geq 0}$ is a Feller process if $\Phi$ is bounded and that the symbol is of the form $p(x,\xi)=\psi(\Phi^\top\!(x)\xi)$, where $\psi$ is the characteristic exponent of the driving L\'evy  process.

\vfill\noindent \emph{MSC 2010:} 60J75; 47G30; 60H20; 60J25; 60G51;
60G17.

\vspace{2mm} \noindent \emph{Keywords:} stochastic differential
equation, L\'evy process, semimartingale, pseudo-differential
operator, Blumenthal-Getoor index, sample path properties

\vspace{2mm}\noindent Submitted to EJP on November 26, 2009, final
version accepted August 12, 2010.

\vspace{2mm}\noindent \emph{Acknowledgements:} We would like to
thank an anonymous referee for carefully reading the manuscript and
offering useful suggestions which helped to improve the paper.

\end{abstract}
\vfill

\pagebreak

\section{Introduction}

Within the last ten years a rich theory concerning the relationship
between Feller processes and their so called symbols which appear in
the Fourier representation of their generator has been developed,
see for example the monographs \cite{jacob1,jacob2,jacob3} by Jacob
or the fundamental contributions by Hoh \cite{hoh98,hoh00,hoh02} and
Ka{\ss}mann \cite{kas09}; see also \cite{bot-sch09} and
\cite{jac-sch-survey} for a survey. In this paper we establish a
stochastic formula to calculate the symbol of a class of Markov
processes which we then apply to the solutions of certain stochastic
differential equations (SDEs). If the coefficient of the SDE is
bounded, the solution turns out to be a Feller process. As there are
different conventions in defining this class of processes in the
literature, let us first fix some terminology: consider a time
homogeneous Markov process $(\Omega,\cf, (\cf_t)_{t\geq
0},(X_t)_{t\geq 0},\Pp^x)_{x\in\bbr^d}$ with state space $\bbr^d$;
we will always assume that the process is normal, i.e.\
$\Pp^x(X_0=x)=1$. As usual, we can associate with a Markov process a
semigroup $(T_t)_{t\geq 0}$ of operators on $B_b(\bbr^d)$ by setting
$$
    T_t u(x):= \Ee^x u(X_t), \quad t\geq 0,\; x\in \bbr^d.
$$
Denote by $C_\infty=C_\infty(\bbr^d,\bbr)$ the space of all
functions $u:\bbr^d\to\bbr$ which are continuous and vanish at
infinity, $\lim_{\abs{x}\to\infty}u(x) =0$; then
$(C_\infty,\nnorm{\cdot}_\infty)$ is a Banach space and $T_t$ is for
every $t$ a contractive, positivity preserving and sub-Markovian
operator on $B_b(\bbr^d)$. We call $(T_t)_{t\geq 0}$ a Feller
semigroup and $(X_t)_{t\geq 0}$ a Feller process if the following
conditions are satisfied:
\begin{description}
    \item[\normalfont (F1)] $T_t:C_\infty \to C_\infty$ for every $t\geq 0$,
    \item[\normalfont (F2)] $\lim_{t\downarrow 0} \nnorm{T_tu-u}_\infty =0$ for every $u\in C_\infty$.
\end{description}

The generator $(A,D(A))$ is the closed operator given by
\begin{gather}\label{generator}
    Au:=\lim_{t \downarrow 0} \frac{T_t u -u}{t} \qquad\text{for\ \ } u\in D(A)
\end{gather}
where the domain $D(A)$ consists of all $u\in C_\infty$ for which the limit \eqref{generator} exists uniformly. Often we have to assume that $D(A)$ contains sufficiently many functions. This is, for example the case, if
\begin{gather}\tag{R}\label{rich}
    C_c^\infty\subset D(A).
\end{gather}
A classical result due to Ph.\ Courr\`ege \cite{courrege} shows that, if \eqref{rich} is fulfilled, $A|_{C_c^\infty}$ is a pseudo differential operator with symbol $-p(x,\xi)$, i.e.\ $A$ can be written as
\begin{gather} \label{pseudo}
    Au(x)= - \int_{\bbr^d} e^{ix^\top\xi} p(x,\xi) \widehat{u}(\xi) \,d\xi, \qquad u\in C_c^\infty
\end{gather}
where $\widehat{u}(\xi)=(2\pi)^{-d}\int e^{-iy^\top\xi}u(y) dy$ denotes the Fourier transform and $p:\bbr^d \times \bbr^d \to \bbc$ is locally bounded and, for fixed $x$, a continuous negative definite function in the sense of Schoenberg in the co-variable $\xi$. This means it admits a L\'evy-Khintchine representation
\begin{align} \label{lkfx}
  p(x,\xi)=
  -i \ell^\top\!(x)  \xi + \frac{1}{2} \xi^\top Q(x) \xi -\int_{y\neq 0} \left( e^{i \xi^\top y} -1 - i \xi^\top y \cdot \I_{\{\abs{y}<1\}}(y)\right)N(x,dy)
\end{align}
where for each $x\in\bbr^d$ $(\ell(x),Q(x),N(x,dy))$ is a L\'evy triplet, i.e.\ $\ell(x)=(\ell^{(j)}(x))_{1\leq j \leq d} \in \bbr^d$, $Q(x)=(q^{jk}(x))_{1\leq j,k \leq d}$ is a symmetric positive semidefinite matrix and $N(x,dy)$ is a measure on $\bbr^d\setminus\{0\}$ such that $\int_{y\neq 0} (1 \wedge |y|^2) \,N(x,dy) < \infty$. The function $p(x,\xi)$ is called the symbol of the operator. For details we refer to the treatise by Jacob \cite{jacob1,jacob2,jacob3}.

Combining \eqref{pseudo} and \eqref{lkfx} the generator $A$ of a Feller process satisfying condition \eqref{rich} can be written in the following way
\begin{align*}
    Au(x)
    &=\ell(x) ^\top\nabla u(x) + \frac{1}{2}\sum_{j,k=1}^d q^{jk}(x) \partial_j \partial_k u(x) \\
    &\qquad+ \int_{y\neq 0} \left( u(x+y) - u(x) - y^\top\nabla u(x) \cdot \I_{B_1(0)}(y) \right) N(x,dy)
\end{align*}
for $u\in C_c^\infty(\bbr^d)$. This is called the integro-differential form of the operator.

An important subclass of Feller processes are L\'evy  processes. These are processes which have stationary and independent increments and which are stochastically continuous. For L\'evy  processes $(Z_t)_{t \geq 0}$ it is well known that the characteristic function can be written in the following way
\begin{align*}
  \Ee^z \left( e^{i (Z_t-z)^\top\xi} \right)
  = \Ee^0 \left( e^{i Z_t^\top\xi} \right) = e^{-t\, \psi(\xi)}
\end{align*}
where $\psi:\bbr^d \to \bbc$ is a continuous negative definite function, i.e.\ it has a L\'evy-Khintchine representation where the L\'evy triplet $(\ell,Q,N)$ does not depend on $x$.

This is closely connected to the following result. Every L\'evy process $(Z_t)_{t \geq 0}$ with L\'evy triplet $(\ell,Q,N)$ has the following L\'evy-It\^o decomposition
\begin{gather} \label{lid}
    Z_t
    = \ell t +\Sigma W_t +\int_{[0,t] \times \{\abs{y}<1\}} y \,\left(\mu^Z(ds,dy)-ds\,N(dy)\right) + \sum_{0<s\leq t} \Delta Z_s \I_{\{\abs{\Delta Z_s}\geq 1 \} }
\end{gather}
where $\ell\in\bbr^d$, $\Sigma$ is the unique positive semidefinite square root of $Q\in\bbr^{d\times d}$, $(W_t)_{t\geq 0}$ is a standard Brownian motion, and $\mu^Z$ is the Poisson point measure given by the jumps of $Z$ whose intensity measure is the L\'evy measure $N$. The second and third terms appearing in \eqref{lid} are martingales, while the other two terms are of finite variation on compacts. Therefore every L\'evy  process is a semimartingale. Note that all four terms are independent.

The generator of a L\'evy  process is given by
\begin{gather}
    Au(x)= - \int_{\bbr^d} e^{ix^\top\xi} \psi(\xi) \widehat{u}(\xi) \,d\xi, \quad u\in C_c^\infty,
\end{gather}
i.e.\ L\'evy  processes are exactly those Feller processes whose
generator has `constant coefficients'.

Every L\'evy process has a symbol (that is: a characteristic
exponent) $\psi$; on the other hand, every $\psi$ and every L\'evy
triplet $(\ell,Q,N)$ defines a L\'evy process. For Feller processes
the situation is different: every Feller process satisfying
\eqref{rich} admits a symbol, but it is not known if every symbol of
the form \eqref{lkfx} yields a process. See
\cite{jac-sch-survey,jacob3} for a survey. On the other hand it is
known that the symbol $p(x,\xi)$ can be used to derive many
properties of the associated process $X$.

In this paper we prove a probabilistic formula for the symbol. We
use this formula to calculate the symbol of the solution of a L\'evy
driven SDE. Let us give a brief outline how the paper is organized:
in Section 2 we introduce the symbol of a Markov process. It turns
out that the symbol which is defined in a probabilistic way
coincides with the analytic (in the sense of pseudo-differential
operators) symbol for the class of Feller processes which satisfy
\eqref{rich}. The main result of the paper can be found in Section
3, where we calculate the symbol of a Feller process, which is given
as the strong solution of a stochastic differential equation. In
Section 4 we consider some extensions; these comprise, in
particular, the case
$$
    dX^x=\Phi(X^x) \,dZ_t + \Psi(X^x)\,dt, \quad X_0^x=x,
$$
which is often used in applications. We close by using the symbol of the process $X^x$ to investigate some of its path properties.

\section{The Symbol of a Markov Process}

\begin{definition}
    Let $X$ be an $\bbr^d$-valued Markov process, which is conservative and normal. Fix a starting point $x$ and define $\sigma=\sigma^x_R$ to be the first exit time from the ball of radius $R > 0$:
    $$
        \sigma:=\sigma^x_R:=\inf\big\{t\geq 0 : \nnorm{X_t^x-x} > R \big\}.
    $$
    The function $p:\bbr^d\times \bbr^d \rightarrow \bbc$ given by
    \begin{gather} \label{symbol}
        p(x,\xi):=- \lim_{t\downarrow 0}\Ee^x \left(\frac{e^{i(X^\sigma_t-x)^\top\xi}-1}{t}\right)
    \end{gather}
    is called the \emph{symbol of the process}, if the limit exists for every $x,\xi\in\bbr^d$ independently of the choice of $R>0$.
\end{definition}

\begin{remark}
(a) In \cite{schnurr-diss} the following is shown even for the larger class of It\^o processes in the sense of \cite{vierleute}: fix $x\in\bbr^d$; if the limit \eqref{symbol} exists for one $R$, then it exists for every $R$ and the limit is independent of $R$.

(b) For fixed $x$ the function $p(x,\xi)$ is negative definite as a function of $\xi$. This can be shown as follows: for every $t > 0$ the function $\xi\mapsto \Ee^x e^{i(X_t^\sigma-x)^\top\xi}$ is the characteristic function of the random variable $X_t^\sigma-x$. Therefore it is a continuous positive definite function. By Corollary 3.6.10 of \cite{jacob1} we conclude that $\xi\mapsto -(\Ee^x e^{i(X_t^\top-x)^\top\xi} -1)$ is a continuous negative definite function. Since the negative definite functions are a cone which is closed under pointwise limits, \eqref{symbol} shows that $\xi\mapsto p(x,\xi)$ is negative definite. Note, however, that $\xi\mapsto p(x,\xi)$ is not necessarily continuous.
\end{remark}

If $X$ is a Feller process satisfying \eqref{rich} the symbol $p(x,\xi)$ is exactly the negative definite symbol which appears in the pseudo differential representation of its generator \eqref{pseudo}. A posteriori this justifies the name.

We need three technical lemmas. The first one is known as Dynkin's formula. It follows from the well known fact that
$$
    M_t^{[u]}:=u(X_t)-u(x)-\int_0^t Au(X_s) \,ds
$$
is a martingale for every $u\in D(A)$ with respect to every $\Pp^x$, $x\in\bbr^d$, see e.g.\  \cite{revuzyor} Proposition VII.1.6.

\begin{lemma} \label{lem:dynkin}
Let $X$ be a Feller process and $\sigma$ a stopping time. Then we have
\begin{gather} \label{dynkin}
    \Ee^x \int_0^{\sigma \wedge t} A u(X_s) \,ds = \Ee^x u(X_{\sigma \wedge t})-u(x)
\end{gather}
for all $t>0$ and $u\in D(A)$.
\end{lemma}

\begin{lemma} \label{lem:limits}
Let $Y^y$ be an $\bbr$-valued process, starting a.s.\ in $y$, which is right continuous at zero and bounded. Then we have
$$
    \frac 1t\, \Ee \int_0^t Y^y_s \,ds \xrightarrow{t\downarrow 0} y.
$$
\end{lemma}

\begin{proof}
It is easy to see that
\begin{align*}
  \abs{\Ee \left(\frac 1t \int_0^t ( Y^y_s-Y^y_0) \,ds\right)}
  \leq \Ee\left(\sup_{0\leq s\leq t}  \abs{Y^y_s-Y^y_0}\right).
\end{align*}
The result follows from the bounded convergence theorem.
\end{proof}

\begin{lemma} \label{lem:folge}
Let  $K\subset \bbr^d$ be a compact set. Let $\chi:\bbr^d \to \bbr$
be a smooth cut-off function, i.e.\ $\chi\in C_c^\infty(\bbr^d)$
with
$$
    \I_{B_1(0)}(y) \leq \chi (y) \leq \I_{B_2(0)}(y)
$$
for $y\in\bbr^d$. Furthermore we define $\chi_n^x(y):=\chi((y-x)/n)$ and $u_n^x(y):= \chi_n^x(y) e^{i y^\top\xi}$. Then we have for all $z\in K$
$$
    \left|u_n^x(z+y)-u_n^x(z)-y^\top\nabla u_n^x(z) \I_{B_1(0)}(y)\right|
    \leq C \cdot \left(\abs{y}^2 \wedge 1 \right).
$$
\end{lemma}

\begin{proof}
Fix a compact set $K\subset\bbr^d$. An application of Taylor's formula shows that there exists a constant $C_K>0$ such that
$$
    \abs{u_n^x(z+y)-u_n^x(z)-y^\top\nabla u_n(z) \I_{B_1(0)}(y) }
    \leq C_K \left(\abs{y}^2 \wedge 1 \right) \sum_{\abs{\alpha}\leq 2} \nnorm{\partial^\alpha u_n^x}_\infty
$$
uniformly for all $z\in K$. By the particular choice of the sequence
$(\chi_n^x)_{n\in\bbn}$ and Leibniz' rule we obtain that
$\sum_{\abs{\alpha}\leq 2} \nnorm{\partial^\alpha u_n^x}_\infty \leq
\sum_{\abs\alpha\leq 2} \nnorm{\partial^\alpha \chi}_\infty
(1+|\xi|^2)$, i.e.\  it is uniformly bounded for all $n\in\bbn$.
\end{proof}

\begin{theorem}\label{theorem:symbol}
    Let $X=(X_t)_{t\geq 0}$ be a conservative Feller process satisfying condition \eqref{rich}. Then the generator $A|_{C_c^\infty}$ is a pseudo-differential operator with symbol $-p(x,\xi)$, cf.\  \eqref{pseudo}. Let
    \begin{gather} \label{stopping}
        \sigma:= \sigma_R^x:= \inf \{ s\geq 0 : \nnorm{X_s-x}>R \}.
    \end{gather}
    If $x\mapsto p(x,\xi)$ is continuous, then we have
    $$
    \lim_{t\downarrow 0}\Ee^x \left( \frac{e^{i(X_t^\sigma-x)^\top\xi} - 1}{t} \right)= -p(x,\xi),
    $$
    i.e.\ the symbol of the process exists and coincides with the symbol of the generator.
\end{theorem}

The assumption that $x\mapsto p(x,\xi)$ is continuous is not a severe restriction. All non-pathological known examples of Feller processes satisfy this condition. It is always fulfilled, if $X$ has only bounded jumps, cf.\ the discussion in \cite{bot-sch09}.
\begin{proof}[Proof of Theorem \ref{theorem:symbol}]
    Let $(\chi_n^x)_{n\in\bbn}$ be the sequence of cut-off functions of Lemma \ref{lem:folge} and we write $e_\xi(x):=e^{i x^\top\xi}$ for $x,\xi\in\bbr^d$. By the bounded convergence theorem and Dynkin's formula \eqref{dynkin} we see
\begin{align*}
    \Ee^x \left( e^{i (X_t^\sigma -x)^\top\xi } -1\right)
    &=\lim_{n\to\infty} \big(\Ee^x \chi_n^x (X_t^\sigma) e_\xi(X_t^\sigma) e_{-\xi}(x) -1 \big) \\
    &=e_{-\xi}(x) \lim_{n\to\infty} \Ee^x \big(\chi_n^x(X_t^\sigma) e_\xi (X_t^\sigma) - \chi_n^x(x) e_\xi(x) \big)\\
    &=e_{-\xi}(x) \lim_{n\to\infty} \Ee^x \int_0^{\sigma \wedge t} A(\chi_n^x  e_\xi)(X_s) \,ds \\
    &=e_{-\xi}(x) \lim_{n\to\infty} \Ee^x \int_0^{\sigma \wedge t} A(\chi_n^x  e_\xi)(X_{s-}) \,ds.
\end{align*}
The last equality follows since we are integrating with respect to
Lebesgue measure and since a c\`adl\`ag process has a.s.\ a
countable number of jumps.   Using Lemma \ref{lem:folge} and the
integro-differential representation of the generator $A$ it is not
hard to see that for all $z\in K := \overline{B_R(x)}$
\begin{align*}
    A(\chi_n e_\xi)(z)
    &\leq c_\chi \left(|\ell(z)| + \frac 12\sum_{j,k=1}^d |q^{jk}(z)| + \int_{y\neq 0} (1\wedge |y|^2)\,N(z,dy)\right) (1+|\xi|^2)\\
    &\leq c_\chi'\sup_{z\in K} \sup_{| \eta |\leq 1} | p(z, \eta )|  (1+|\xi|^2) ;
\end{align*}
the last estimate follows with (some modifications of) techniques
from \cite{schilling98} which we will, for the readers' convenience,
work out in the Appendix. Being the symbol of a Feller process,
$p(x,\xi)$ is locally bounded (cf. \cite{courrege} Th\'eor\`eme
3.4). By definition of the stopping time $\sigma$ we know that for
all $s\leq\sigma\wedge t$ we have $z =
X_{s-}\in\overline{B_R(x)}=K$. Therefore, the integrand $A(\chi_n^x
e_\xi)(X_{s-})$, $s\leq \sigma\wedge t$ appearing in the above
integral is bounded and we may use the dominated convergence theorem
to interchange limit and integration. This yields
\begin{align*}
    \Ee^x \left( e^{i (X_t^\sigma -x)^\top\xi } -1\right)
    &=e_{-\xi}(x) \Ee^x \int_0^{\sigma \wedge t} \lim_{n\to \infty} A(\chi_n^x e_\xi)(z) |_{z=X_{s-}} \,ds \\
    &=-e_{-\xi}(x) \Ee^x \int_0^{\sigma \wedge t}  e_\xi(z) p(z,\xi)|_{z=X_{s-}} \,ds.
\end{align*}
The   second  equality follows from \cite{courrege} Sections 3.3 and
3.4. Therefore,
\begin{align*}
    \lim_{t\downarrow 0} \frac{\Ee^x \left(e^{i(X_t^\sigma - x)^\top\xi}-1\right)}{t}
    &=-e_{-\xi}(x) \lim_{t \downarrow 0} \Ee^x\left( \frac 1t \int_0^t e_\xi(X_{s-}^\sigma) p(X_{s-}^\sigma,\xi) \I_{\llbracket 0,\sigma\llbracket}(s) \,ds\right) \\
    &=-e_{-\xi}(x) \lim_{t \downarrow 0} \Ee^x\left( \frac 1t \int_0^t e_\xi(X_{s}^\sigma) p(X_{s}^\sigma,\xi) \I_{\llbracket 0,\sigma\llbracket}(s) \,ds\right)
\end{align*}
since we are integrating with respect to Lebesgue measure. The process $X^\sigma$ is bounded on the stochastic interval $\llbracket 0,\sigma\llbracket$ and $x\mapsto p(x,\xi)$ is continuous for every $\xi\in\bbr^d$. Thus, Lemma \ref{lem:limits} is applicable and gives
\begin{gather*}
    \lim_{t\downarrow 0} \frac{\Ee^x \left(e^{i(X_t^\sigma - x)^\top\xi}-1\right)}{t}
    =-e_{-\xi}(x) e_{\xi}(x) p(x,\xi)
    =-p(x,\xi).
    \qedhere
\end{gather*}
\end{proof}

Theorem \ref{theorem:symbol} extends an earlier result from \cite{schilling98pos} where additional assumptions are needed for $p(x,\xi)$. An extension to It\^o processes is contained in \cite{schnurr-diss}.

\section{Calculating the Symbol}

Let $Z=(Z_t)_{t\geq 0}$ be an $n$-dimensional L\'evy process starting at zero with symbol $\psi$ and consider the following SDE
\begin{align} \label{sde}
  dX_t^x&=\Phi(X_{t-}^x) \,dZ_t \\
  X_0^x&=x \nonumber
\end{align}
where $\Phi: \bbr^d \to \bbr^{d \times n}$ is locally Lipschitz continuous and satisfies the following linear growth condition: there exists a $K>0$ such that for every $x\in\bbr^d$
\begin{align} \label{growth}
    \abs{\Phi(x)}^2 \leq K(1+\abs{x}^2).
\end{align}
Since $Z$ takes values in $\bbr^n$ and the solution $X^x$ is $\bbr^d$-valued, \eqref{sde} is a shorthand for the system of stochastic integral equations
$$
    X^{x,(j)} =  x^{(j)} +  \sum_{k=1}^n \int\Phi(X_-)^{jk} \,dZ^{(k)},\quad j=1,\ldots,d.
$$

A minor technical difficulty arises if one takes the starting point into account and if all processes $X^x$ should be defined on the same probability space. The original space $(\Omega, \cf, (\cf_t)_{t\geq 0}, \Pp)$ where the driving L\'evy process is defined is, in general, too small as a source of randomness for the solution processes. We overcome this problem by enlarging the underlying stochastic basis as in \cite{protter}, Section 5.6:
\begin{align*}
    \overline{\Omega}&:= \bbr^d \times \Omega,
    & \Pp^x&:=\varepsilon_x \times \Pp,\quad x \in \bbr^d,\\
    \overline{\cf_t^0}&:= \cb^d\otimes \cf_t &
    \overline{\cf_t} &:=\bigcap_{u>t}\overline{\cf_u^0}
\end{align*}
where $\varepsilon_x$ denotes the Dirac measure in $x$. A random variable $Z$ defined on $\Omega$ is considered to be extended automatically to $\overline{\Omega}$ by $Z(\overline{\omega})=Z(\omega)$, for $\overline{\omega}=(x,\omega)$.

It is well known that under the local Lipschitz and linear growth
conditions imposed above, there exists a unique conservative
solution of the SDE \eqref{sde}, see e.g.\ \cite{metivier} Theorem
34.7 and Corollary 35.3.

\begin{theorem}\label{sde-symbol}
    The unique strong solution of the SDE \eqref{sde} $X_t^x(\omega)$ has the symbol $p:\bbr^d\times \bbr^d \to \bbc$ given by
    $$
        p(x,\xi)=\psi(\Phi^\top\!(x)\xi)
    $$
    where $\Phi$ is the coefficient of the SDE and $\psi$ the symbol of the driving L\'evy process.
\end{theorem}

\begin{proof}
To keep notation simple, we give only the proof for $d=n=1$. The
multi-dimensional version is proved along the same lines, the only
complication being notational; a detailed account is given in
\cite{schnurr-diss}. Let $\sigma$ be the stopping time given by
\eqref{stopping}. Fix $x,\xi \in \bbr$. We apply It\^o's formula for
jump processes to the function $e_\xi(\cdot - x) =
\exp(i(\cdot-x)\xi)$:
\begin{equation}\label{threeterms}\begin{aligned}
    \frac 1t\,\Ee^x \big( &e^{i(X^\sigma_t-x)\xi} -1\big)\\
    &=\frac 1t\,\Ee^x \bigg( \int_{0+}^t i \xi\,e^{i (X^\sigma_{s-}-x) \xi} \,dX^\sigma_s
    - \frac 12 \int_{0+}^{t} \xi^2\, e^{i (X^\sigma_{s-}-x) \xi} \,d[X^\sigma,X^\sigma]_s^c\\
    &\qquad+e^{-ix \xi} \sum_{0<s \leq t} \left(e^{i X^\sigma_s \xi} - e^{i X^\sigma_{s-} \xi} -i\xi
      e^{iX^\sigma_{s-} \xi}\Delta X^\sigma_s\right) \bigg).
\end{aligned}\end{equation}
For the first term we get
\begin{align}
  \frac 1t\,\Ee^x &\int_{0+}^{t} \left(i \xi\, e^{i (X^\sigma_{s-}-x) \xi}\right) \,dX^\sigma_s \notag\\
  &= \frac 1t\,\Ee^x \int_{0+}^{t} \left(i \xi \, e^{i (X^\sigma_{s-}-x) \xi}\right) d\left(\int_0^s\Phi(X_{r-})\I_{\llbracket 0,\sigma\rrbracket}(\cdot,r) \,dZ_r\right)\notag\\
  &= \frac 1t\, \Ee^x \int_{0+}^{t} \left(i \xi \, e^{i (X^\sigma_{s-}-x) \xi}
        \Phi(X_{s-})\I_{\llbracket 0,\sigma\rrbracket}(\cdot,s) \right)\,dZ_s\notag\\
  &= \frac 1t\, \Ee^x \int_{0+}^{t} \left(i \xi \, e^{i (X^\sigma_{s-}-x) \xi}
        \Phi(X_{s-})\I_{\llbracket 0,\sigma\rrbracket}(\cdot,s)\right) \,d(\ell s)\label{drift}\\
  &\qquad  + \frac 1t\, \Ee^x \int_{0+}^{t} \left(i \xi \, e^{i (X^\sigma_{s-}-x) \xi}
        \Phi(X_{s-})\I_{\llbracket 0,\sigma\rrbracket}(\cdot,s)\right) d\left(\sum_{0<r\leq s} \Delta Z_r \I_{\{\abs{\Delta Z_r}\geq 1 \}}\right)\label{jumps}
\end{align}
where we have used the L\'evy-It\^o decomposition \eqref{lid}. Since the integrand is bounded, the martingale terms of \eqref{lid} yield martingales whose expected value is zero.

First we deal with \eqref{jumps} containing the big jumps. Adding this integral to the third expression on the right-hand side of \eqref{threeterms} we obtain
\begin{align*}
    \frac 1t\,\Ee^x & \sum_{0<s \leq t} \Big(e^{i(X^\sigma_{s-}-x)\xi}
    \left( e^{i \Phi(X_{s-})\Delta Z_s \xi} -1- i\xi \Phi(X_{s-})\Delta Z_s  \I_{\{\abs{\Delta X_s}< 1
    \}}\right) \I_{\llbracket 0,\sigma\rrbracket}(\cdot,s) \Big)  \\
    &= \frac 1t\,\Ee^x \int_{]0,t]\times \bbr \setminus \{ 0\}}  H_{x,\xi}(\cdot\,;s-,y)
    \mu^X(\cdot\,;ds,dy) \\
    &= \frac 1t\,\Ee^x \int_{]0,t]\times \bbr \setminus \{ 0\}} H_{x,\xi}(\cdot\,;s-,y)
    \nu(\cdot\,;ds,dy) \\
  & \xrightarrow{t\downarrow 0}  \int_{\bbr \setminus \{ 0\}}
    \left( e^{i \Phi(x)y \xi} -1- i\xi \Phi(x)y  \I_{\{\abs{y}< 1 \}} \right) N(dy)
\end{align*}
where we have used Lemma \ref{lem:limits} and the shorthand
$$
    H_{x,\xi}(\omega;s,y):=
    e^{i(X^\sigma_s(\omega)-x)\xi} \left( e^{i \Phi(X_{s}(\omega))y \xi} -1 - i\xi \Phi(X_{s}(\omega))y \I_{\{\abs{y}< 1\}}\right) \I_{\llbracket 0,\sigma\rrbracket}(\omega,s).
$$
The calculation above uses some well known results about integration with respect to integer valued random measures, see \cite{ikedawat} Section II.3, which allow us to integrate `under the expectation' with respect to the compensating measure $\nu(\cdot\,;ds,dy)$ instead of the random measure itself. In the case of a L\'evy process the compensator is of the form $\nu(\cdot\,;ds,dy)=N(dy) \,ds $, see \cite{ikedawat} Example II.4.2.

For the drift part \eqref{drift} we obtain
\begin{align*}
    \frac 1t\, \Ee^x &\int_{0+}^{t} \left(i \xi \cdot e^{i (X^\sigma_{s-}-x) \xi}
      \Phi(X_{s-})\I_{\llbracket 0,\sigma\rrbracket}(\cdot,s)\ell \right) \,ds \\
    &= i \xi\ell  \cdot \Ee^x \frac 1t\int_{0}^{t} \left(e^{i (X^\sigma_{s}-x) \xi}
      \Phi(X_{s}) \I_{\llbracket 0,\sigma\llbracket}(\cdot,s)\right)\,ds
    \xrightarrow{t\downarrow 0} i \xi \ell \Phi(x)
\end{align*}
where we have used Lemma \ref{lem:limits} in a similar way as in the proof of Theorem \ref{theorem:symbol}.

We can deal with the second expression on the right-hand side of
\eqref{threeterms} in a similar way, once we have worked out the
square bracket of the process.
\begin{align*}
  [X^\sigma,X^\sigma]_t^c =([X,X]_t^c)^\sigma
  &= \left(\left[{\textstyle \int_0^\cdot \Phi(X_{r-}) dZ_r,\; \int_0^\cdot \Phi(X_{r-}) dZ_r}\right]_t^c\right)^\sigma\\
  &=\int_0^t \Phi(X_{s-})^2 \I_{\llbracket 0,\sigma\rrbracket}(\cdot,s) \,d[Z,Z]_s^c  \\
  &=\int_0^t \Phi(X_{s-})^2 \I_{\llbracket 0,\sigma\rrbracket}(\cdot,s) \,d(Q s)
\end{align*}
Now we can calculate the limit for the second term
\begin{align*}
  \frac{1}{2t}\,\Ee^x &\int_{0+}^{t} \left(-\xi^2 \, e^{i (X^\sigma_{s-}-x) \xi}\right) \,d[X^\sigma,X^\sigma]_s^c \\
  &= \frac{1}{2t}\, \Ee^x \int_{0+}^{t} \left(-\xi^2\, e^{i (X^\sigma_{s-}-x) \xi} \right)\,d\left(\int_{0}^s (\Phi(X_{r-}))^2 \I_{\llbracket 0,\sigma\rrbracket}(\cdot,r) Q \,dr\right) \\
  &=-\frac{1}{2}\,\xi^2 Q\, \Ee^x \left(\frac 1t \int_{0}^t \left( e^{i(X^\sigma_s-x)\xi}\Phi(X_{s})^2 \I_{\llbracket 0,\sigma\llbracket}(\cdot,s) \right) \,ds\right) \\
  &\xrightarrow{t\downarrow 0} -\frac{1}{2}\, \xi^2 Q\, \Phi(x)^2.
\end{align*}
In the end we obtain
\begin{align*}
  p(x,\xi)
  &= -i \ell  (\Phi(x)\xi) + \frac{1}{2} (\Phi(x)\xi)Q (\Phi(x)\xi) \\
  &\qquad -\int_{y\neq 0} \left(e^{i (\Phi(x)\xi)y} -1 - i (\Phi(x)\xi)y \cdot \I_{\{\abs{y}<1\}}(y)\right) N(dy)\\
  &=\psi(\Phi(x)\xi).
\end{align*}
Note that in the multi-dimensional case the matrix $\Phi(x)$ has to be transposed, i.e.\ the symbol of the solution is $\psi(\Phi^\top\!(x)\xi)$.
\end{proof}

Theorem \ref{sde-symbol} shows that it is possible to calculate the symbol, even if we do not know whether the solution process is a Feller process. However, most of the interesting results concerning the symbol of a process are restricted to Feller processes. Therefore it is interesting to have conditions guaranteeing that the solution of \eqref{sde} is a Feller process.
\begin{theorem}\label{feller-solution}
    Let $Z$ be a $d$-dimensional L\'evy processes such that $Z_0=0$. Then the solution of \eqref{sde} is a strong Markov process under each $\Pp^x$. Furthermore the solution process is time homogeneous and the transition semigroups coincide for every $\Pp^x$, $x\in \bbr^d$.
\end{theorem}

\begin{proof}
    See Protter \cite{protter} Theorem V.32 and \cite{protter77} Theorem (5.3). Note that Protter states the theorem only for the special case where the components of the process are independent. However the independence is not used in the proof.
\end{proof}

Some lengthy calculations lead from Theorem \ref{feller-solution} directly to the following result which can be found in \cite{applebaum} Theorem 6.7.2 and, with an alternative proof, in \cite{schnurr-diss} Theorem 2.49.
\begin{corollary}
    If the coefficient $\Phi$ is bounded, the solution process $X^x_t$ of the SDE given by \eqref{sde} is a Feller process.
\end{corollary}

\begin{remark}
    In \cite{schnurr-diss} it is shown that if $\Phi$ is not bounded the solution of \eqref{sde} may fail to be a Feller process. Consider the stochastic integral equation
    $$
        X_t=x-\int_0^t X_{s-} \,dN_s
    $$
    where $N=(N_t)_{t\geq 0}$ is a standard Poisson process. The solution process starts in $x$, stays there for an exponentially distributed waiting time (which is independent of $x$) and then jumps to zero, where it remains forever. There exists a time $t_0>0$ for which $\Pp^x(X_{t_0}=x) = \Pp^x(X_{t_0}=0) = 1/2$. For a function $u\in C_c(\bbr)$ with the property $u(0)=1$ we obtain
    $$
        \Ee^x(u(X_{t_0})) = \frac{1}{2} \hspace{10mm} \text{ for every } x \notin\text{supp}u.
    $$
    Therefore $T_{t_0} u$ does not vanish at infinity.
\end{remark}

Next we show that the solution of the SDE satisfies condition \eqref{rich} if $\Phi$ is bounded.
\begin{theorem}\label{thm:rich}
Let $\Phi$ be bounded and locally Lipschitz continuous. In this case
the solution $X^x_t$ of the SDE
$$
    X_t=x+\int_0^t \Phi(X_{s-}) \,dZ_s, \quad x\in\bbr^d,
$$
fulfills condition \eqref{rich}, i.e.\ the test functions are contained in the domain $D(A)$ of the generator $A$.
\end{theorem}

\begin{proof}
Again we only give the proof in dimension one. The multi-dimensional version is similar. Let $u\in C_c^\infty(\bbr)$. By It\^o's formula we get
\begin{align*}
    D_t
    &:= \frac{\Ee^x u(X_t) - u(x)}{t}\\
    &=\frac 1t\Ee^x (u(X_t)-u(x))\\
    &=\frac 1t\,\Ee^x \bigg( \int_{0+}^t u'(X_{s-}) \,dX_s +\frac{1}{2} \int_{0+}^t u''(X_{s-}) \,d[X,X]_s^c \\
    &\qquad+\sum_{0<s\leq t} \big( u(X_s) -u(X_{s-}) - u'(X_{s-}) \Delta X_s\big)\bigg).
\end{align*}
Since $X_t=x+\int_0^t\Phi(X_{s-})\,dZ_s$ we obtain
\begin{align*}
    D_t
    &=\frac 1t\,\Ee^x \bigg(\int_{0+}^t u'(X_{s-}) \Phi(X_{s-}) \,dZ_s +\frac{1}{2} \int_{0+}^t u''(X_{s-}) \Phi(X_{s-})^2 \,d[Z,Z]_s^c\\
    &\quad + \int_{y\neq 0} \int_0^t  \Big(u\left(X_{s-}+\Phi(X_{s-})y\right)-u(X_{s-})-u'(X_{s-})\Phi(X_{s-})y\Big) \mu^Z(\cdot\,;ds,dy)\bigg)
\end{align*}
where $\mu^Z$ is the random measure given by the jumps of the L\'evy process $Z$. Next we use the L\'evy-It\^o decomposition $Z$ in the first term. The expected value of the integral with respect to the martingale part of $Z$ is zero, since the integral
$$
    \int_0^t u'(X_{s-})\Phi(X_{s-}) \, d\left(\Sigma W_t + \int_{[0,t] \times \{\abs{y}<1\}} y\,\left(\mu^Z(ds,dy)-ds\,N(dy)\right)\right)
$$
is an $L^2$-martingale. Therefore we obtain
\begin{align*}
    D_t
    &=\frac 1t\,\Ee^x\int_{0+}^t u'(X_{s-})\Phi(X_{s-})  \,d\left(\ell t +\sum_{0<r\leq s} \Delta Z_r \I_{\{\abs{\Delta Z_r}\geq 1 \}} \right) \\
    &\quad + \frac{1}{2}\frac 1t\,\Ee^x \int_{0+}^t u'' (X_{s-}) \Phi(X_{s-}) \,d(\Sigma^2 s)\\
    &\quad + \frac 1t\,\Ee^x \int_{y\neq 0} \int_0^t \Big(  u\left(X_{s-}+\Phi(X_{s-})y\right)- u(X_{s-})-u'(X_{s-})\Phi(X_{s-})y \Big) \mu^Z(\cdot\,;ds,dy).
\end{align*}
We write the jump part of the first term as an integral with respect to $\mu^Z$ and add it to the third term. The integrand
$$
    H(\cdot\,;s,y) := u\big(X_{s-}+\Phi(X_{s-})y\big)-u(X_{s-})-u'(X_{s-})\Phi(X_{s-})y \, \I_{\{\abs{y}<1\}}
$$
is in the class $F_p^1$ of Ikeda and Watanabe, \cite{ikedawat} Section II.3, i.e.\ it is predictable and
$$
  \Ee\left( \int_0^t \int_{y\neq 0} \abs{H(\cdot\,;s,y)}\, \nu(\cdot,ds,dy)\right) <\infty
$$
where $\nu$ denotes the compensator of $\mu^X$. Indeed, the measurability criterion is fulfilled because of the left-continuity of $H(\cdot\,;s,\cdot)$, the integrability follows from
\begin{align*}
    \big| u\big( &X_{s-}+\Phi(X_{s-})y\big)-u(X_{s-})-u'(X_{s-})\Phi(X_{s-})y\, \I_{\{\abs{y}<1\}}\big|  \\
    &\leq \abs{\left\{u\big(X_{s-}+\Phi(X_{s-})y\big) -u(X_{s-})- u'(X_{s-}) \Phi(X_{s-})y\right\}\I_{\{\abs{y}<1\}} } + 2 \nnorm{u}_\infty\I_{\{\abs{y}\geq 1\}}\\
    &\leq \frac{1}{2}\,y^2\,\Phi(X_{s-})^2\nnorm{u''}_\infty \I_{\{\abs{y} < 1\}}  + 2 \nnorm{u}_\infty \I_{\{\abs{y}\geq 1\}}\\
    &\leq \left(2\vee \nnorm{\Phi}_\infty^2\right) \left(y^2 \wedge 1\right) \left(\nnorm{u}_\infty + \nnorm{u''}_\infty\right)
\end{align*}
where we used a Taylor expansion for the first term. Therefore $H \in F_p^1$ and we can, `under the expectation', integrate with respect to the compensator of the random measure instead of the measure itself, see \cite{ikedawat} Section II.3. Thus,
\begin{align*}
    D_t
    &=\frac 1t\,\Ee^x\int_{0+}^t u'(X_{s-}) \Phi(X_{s-})\ell  \,ds +\frac{1}{2t}\,\Ee^x \int_{0+}^t u'' (X_{s-})  \Phi(X_{s-})\Sigma^2 \,ds\\
    &\quad + \frac 1t\,\Ee^x \int_{y\neq 0} \int_0^t \left( u(X_{s-}+\Phi(X_{s-})y)-u(X_{s-})-u'(X_{s-})\Phi(X_{s-})y \I_{\{\abs{y}<1\}} \right) \,ds\,N(dy).
\end{align*}
Since we are integrating with respect to Lebesgue measure and since
the paths of a c\`adl\`ag process have only countably many jumps we
get
\begin{align*}
    D_t
    &=\frac 1t\,\Ee^x\int_{0}^t u'(X_{s}) \Phi(X_{s})\ell  \,ds + \frac{1}{2t}\,\Ee^x \int_{0}^t u'' (X_{s})  \Phi(X_{s})\Sigma^2 \,ds\\
    &\quad + \frac 1t\,\Ee^x  \int_0^t \int_{y\neq 0}\left(u(X_{s}+\Phi(X_{s})y)-u(X_{s})-u'(X_{s})\Phi(X_{s})y\I_{\{\abs{y}<1\}} \right) N(dy)\,ds.
\end{align*}
The change of the order of integration is again justified by the estimate of $\abs{H}$. By Lemma \ref{lem:limits} we see that
\begin{align*}
    \frac{\Ee^x u(X_t) - u(x)}{t}
    &\xrightarrow{t \downarrow 0} \ell u'(x) \Phi(x) +\frac{1}{2} \Sigma^2 u''(x) \Phi(x)^2 \\
    &\qquad+ \int_{y\neq 0} \left( u(x+\Phi(x)y)-u(x) -u'(x)\Phi(x)y \cdot \I_{\{\abs{y}<1\}} \right) N(dy).
\end{align*}
As a function of $x$, the limit is continuous and vanishes at infinity. Therefore the test functions are contained in the domain, cf.\ Sato \cite{sato} Lemma 31.7.
\end{proof}

\begin{remark}
In the one-dimensional case the following weaker condition is
sufficient to guarantee that the test functions are contained in the
domain of the solution. Let $\Phi$ be locally Lipschitz continuous
satisfying \eqref{growth} and assume that
\begin{gather}\label{weakercond}
    x\mapsto \sup_{\lambda \in ]0,1[}\frac{1}{x+\lambda \Phi(x)} \in C_\infty(\bbr).
\end{gather}
The products $u'\Phi$ and $u''\Phi$ are bounded for every continuous $\Phi$, because $u$ has compact support. The only other step in the proof of Theorem \ref{thm:rich} which requires the boundedness of $\Phi$ is the estimate of $\abs{H}$ in order to get $H\in F_p^1$.

However, \eqref{weakercond} implies that for every $r>0$ there exists some $R>0$ such that
\begin{gather}\label{weakercond2}
    \abs{x+\lambda \Phi(x)} > r \quad\text{for all\ \ } \abs x > R,\; \lambda\in ]0,1[.
\end{gather}
Therefore, see the proof of Theorem \ref{thm:rich}, we can use Taylor's formula to get
\begin{align*}
    |H(\cdot\,;x,y)|  \I_{\{\abs{y}<1\}}
    &= \abs{\left\{ u\big(X_{s-}+\Phi(X_{s-})y\big)-u(X_{s-})-u'(X_{s-})\Phi(X_{s-})y\right\} \I_{\{\abs{y}<1\}}} \\
    &\leq \abs{\frac{1}{2}\,y^2\,\Phi\big(X_{s-}\big)^2 \,u''\big(X_{s-} +\vartheta y \Phi(X_{s-})\big) \I_{\{\abs{y}<1\}}}
\end{align*}
for some $\vartheta \in ]0,1[$. Set $\lambda:=\vartheta\cdot y$ and pick $r$ such that $\text{supp}u''\subset \overline{B_{r}(0)}$; then \eqref{weakercond2} shows that $\Phi(X_{s-})^2\, u''(X_{s-}+\vartheta y \Phi(X_{s-}))$ is bounded.
\end{remark}

Combining our results, we obtain the following existence result for Feller processes.
\begin{corollary}\label{feller-existence}
    For every negative definite symbol having the following structure
$$
    p(x,\xi)=\psi(\Phi^\top\!(x)\xi) \,
$$
    where $\psi:\bbr^n\to \bbc$ is a continuous negative definite function and $\Phi: \bbr^d \to \bbr^{d \times n}$ is bounded and Lipschitz continuous, there exists a unique Feller process $X^x$. The domain $D(A)$ of the infinitesimal generator $A$ contains the test functions $C_c^\infty = C_c^\infty(\bbr^n)$ and $A|_{C_c^\infty}$ is a pseudo-differential operator with symbol $-p(x,\xi)$.
\end{corollary}

We close this section by mentioning that in a certain sense our investigations of the SDE \eqref{sde} cannot be generalized. For this we cite the following theorem by Jacod and Protter \cite{jacodprotter} which is a converse to our above considerations.

\begin{theorem}
    Let $(\Omega, \cf, (\cf_t)_{t\geq 0}, \Pp)$ be a filtered probability space with a semimartingale $Z$. Let $f\in B(\bbr)$ such that $f$ is never zero and is such that for every $x\in \bbr$ the equation \eqref{sde} has a unique (strong) solution $X^x$. If each of the processes $X^x$ is a time homogeneous Markov process with the same transition semigroup then $Z$ is a L\'evy process.
\end{theorem}

\section{Examples}

In the case $d=1$ we obtain results for various processes which are used most often in applications:

\begin{corollary}
Let $Z^1,\ldots,Z^n$ be independent L\'evy processes with symbols (i.e.\ characteristic exponents) $\psi_1,\ldots,\psi_n$ and let $\Phi^1,\ldots,\Phi^n$ be bounded and Lipschitz continuous functions on $\bbr$. Then the SDE
\begin{equation}\begin{aligned}
    dX_t^x &= \Phi^1(X_{t-}^x) \,dZ^1_t +\cdots+ \Phi^n(X_{t-}^x) \,dZ^n_t\\
    X_0^x  &= x
\end{aligned}\end{equation}
has a unique solution $X^x$ which is a Feller process and admits the symbol
$$
    p(x,\xi)=\sum_{j=1}^n \psi_j(\Phi^j(x)\xi), \quad x,\xi\in\bbr.
$$
\end{corollary}

\begin{proof}
This follows directly from the multi-dimensional case of Theorem 3.1 if one writes the SDE \eqref{sde} in the form
\begin{align*}
    dX_t
    &=(\Phi^1,\ldots,\Phi^n)(X_{t-}) \,d\!\begin{pmatrix} Z^1_t \\ \vdots \\Z^n_t\end{pmatrix} \\
    X_0^x
    &=x.
\qedhere\end{align*}
\end{proof}

\begin{example}[L\'evy  plus Lebesgue]
Let $\Phi,\Psi:\bbr \to \bbr$ be bounded and Lipschitz continuous and $(Z_t)_{t\geq 0}$ be a one-dimensional L\'evy  process with symbol $\psi$. The unique solution process $X^x$ of the SDE
\begin{equation}\begin{aligned}
  dX_t^x
  &= \Phi(X_{t-}^x) \,dZ_t + \Psi(X_{t-}^x) \,dt\\
  X_0^x
  &= x
\end{aligned}\end{equation}
has the symbol $p(x,\xi)=\psi(\Phi(x)\xi)-i(\Psi(x) \xi)$. Note that the driving processes $dX_t^x$ and $dt$ are independent since the latter is deterministic.
\end{example}

\begin{example}[Wiener plus Lebesgue]
Let $\Phi,\Psi:\bbr \to \bbr$ be bounded and Lipschitz continuous and $(W_t)_{t\geq 0}$ be a one-dimensional Brownian motion. The unique solution process $X^x$ of the SDE
\begin{equation}\begin{aligned}
  dX_t^x &= \Phi(X_{t-}^x) \,dW_t + \Psi(X_{t-}^x) \,dt\\
  X_0^x  &= x
\end{aligned}\end{equation}
has the symbol $p(x,\xi)=\abs{\Phi(x)}^2 \abs{\xi}^2 - i \,\Psi(x) \xi$.
\end{example}

\begin{example}[Symmetric $\alpha$-stable] Let $(Z^j_t)_{t\geq 0}$, $j=1,\ldots, n$, be independent symmetric one-dimensional $\alpha_j$-stable L\'evy processes, i.e.\ the characteristic exponents are of the form $\psi(\xi)=\abs{\xi}^{\alpha_j}$ with $\alpha_j \in (0,2]$, and let $\Phi_j:\bbr \to \bbr$ be bounded and Lipschitz continuous. The unique solution process $X^x$ of the SDE
\begin{equation}\begin{aligned}
  dX_t^x
  &= \Phi_1(X_{t-}^x) \,dZ^1_t +\cdots+\Phi_n(X_{t-}^x) \,dZ^n_t\\
  X_0^x
  &= x
\end{aligned}\end{equation}
has the symbol $p(x,\xi)= \sum_{j=1}^n \abs{\Phi_j(x)}^{\alpha_j} \cdot \abs{\xi}^{\alpha_j}$.
\end{example}

\section{Some Applications}

Using the symbol of a Feller Process it is possible to introduce so-called indices which are generalizations of the Blumenthal-Getoor index $\beta$ for a L\'evy process, see e.g.\ \cite{schilling98,jac-sch-survey}. These indices can be used to obtain results about the global behaviour and the paths of the process.

\begin{remark}\label{rem:gdef}
    It is shown in \cite{jac-sch-lkf} Lemma 5.2,  see also Lemma \ref{l1} in the Appendix,  that
    $$
        \frac{\abs{y}^2}{1+\abs{y}^2} = \int_{\bbr^d\setminus\{0\}} \left(1 - \cos(y^\top\rho)\right) g_d(\rho) \,d\rho
    $$
    where
    \begin{gather} \label{gdef}
     g_d(\rho)=\frac{1}{2} \int_0^\infty (2\pi \lambda)^{-d/2} e^{-\abs{\rho}^2 / (2\lambda)} e^{-\lambda/2} \,d\lambda,
     \quad \rho\in\bbr^d\setminus\{0\}.
    \end{gather}
    It is straightforward to see that $\int_{\bbr^d\setminus\{0\}} |\rho|^j\,g(\rho)\,d\rho < \infty$ for all $j=0,1,2,\ldots$
\end{remark}

Let $p(x,\xi)$, $x,\xi\in\bbr^d$, be the symbol of a Markov process. Set
\begin{equation}\label{H-function}
    H(x,R):= \sup_{\abs{y-x}\leq 2R} \sup_{\abs{{\mathsf{e}}}\leq 1} \left( \int_{-\infty}^\infty \Re p\left(y, \frac{\rho {\mathsf{e}}}{R}  \right) g(\rho) \,d \rho + \abs{ p\left(y,\frac{{\mathsf{e}}}{R}\right)}\right)
\end{equation}
with the function $g=g_1$ from Remark \ref{rem:gdef}. If the symbol satisfies the following sector-type condition, $\abs{\Im p(x,\xi)}\leq c_0 \cdot \Re p(x,\xi)$, we set
\begin{equation}\label{h-function}
    h(x,R):= \inf_{\abs{y-x}\leq 2R} \sup_{\abs{{\mathsf{e}}}\leq 1} \Re p\left(y,\frac{{\mathsf{e}}}{4 \kappa R} \right)
\end{equation}
where $\kappa:= ( 4 \arctan (1 / 2c_0))^{-1}$.

\begin{definition} \label{def:indices}
    Let $p(x,\xi)$, $x,\xi\in\bbr^d$ be the symbol of a Markov process. Then
    \begin{gather*}
        \beta^x_\infty
        :=  \inf\left\{\lambda > 0 : \limsup_{R\to 0} R^\lambda H(x,R) = 0 \right\}
    \end{gather*}
    is the generalized \emph{upper index at infinity}. If $|\Im p(x,\xi)| \leq c_0\cdot\Re p(x,\xi)$, then
    \begin{gather*}
        \beta_0
        :=  \sup\left\{ \lambda \geq 0 : \limsup_{R\to\infty} R^\lambda \sup_{x\in\bbr^d} H(x,R) = 0 \right\}
    \end{gather*}
    is the generalized \emph{upper index at zero}.
\end{definition}
In a similar fashion one can define lower versions of these indices using the function $h(x,R)$ and $\liminf$ which are useful for fine properties of the sample paths, cf.\ \cite{schilling98} for details. Here we restrict our attention to $\beta_\infty^x$ and $\beta_0$. The next lemma helps to simplify the Definition \ref{def:indices}.

\begin{lemma}\label{lem:asymp}
$\displaystyle
    H(x,R) \asymp \sup_{|y-x|\leq 2R}\sup_{|\mathsf{e}|\leq 1} \left|p\left(y,\frac{\mathsf{e}} R\right)\right|
$ for all $R>0$ and $x\in\bbr^d$.
\end{lemma}
\begin{proof}
    The estimate $H(x,R)\geq \sup_{|y-x|\leq 2R} \sup_{|\mathsf{e}|\leq 1} \left|p\left(y,\frac{\mathsf{e}} R\right)\right|$ follows immediately from \eqref{H-function}.
    Since $\xi\mapsto p(x,\xi)$ is negative definite, the square root is subadditive, cf.\ \cite{bergforst},
    $$
        \sqrt{|p(x,\xi+\eta)|}
        \leq \sqrt{|p(x,\xi)|} + \sqrt{|p(x,\eta)|}
    $$
    and we conclude that for all $R, \rho>0$ and $y\in\bbr^d$
    \begin{align*}
        \sup_{|\mathsf{e}|\leq 1}\sqrt{\left|p\left(y,\rho\,\frac{\mathsf{e}} R\right)\right|}
        &\leq \sup_{|\mathsf{e}|\leq 1} \left(\sqrt{ \left|p\Big(y, \frac{\mathsf{e}} R\Big)\right|}
        \:\lfloor\rho\rfloor
        + \sqrt{\left|p\Big(y,(\rho-\lfloor\rho\rfloor)\frac{\mathsf{e}}R\Big)\right|}\right)\\
        &\leq 2\,\sup_{|\mathsf{e}|\leq 1}\sqrt{\left|p\left(y,\,\frac{\mathsf{e}} R\right)\right|}\,\big(1+\rho\big).
    \end{align*}
    Since $\int_{-\infty}^\infty (1+\rho)^2\,g(\rho)\,d\rho < \infty$, the lemma follows.
\end{proof}

\begin{example}
(a) For a $d$-dimensional symmetric $\alpha$-stable L\'evy process the symbol is given by $p(x,\xi)=\psi(\xi)=|\xi|^\alpha$, $x,\xi\in\bbr^d$. In this case $\beta_\infty^x = \alpha$ and $\beta_0=\alpha$.

(b) The symbol $p(x,\xi)=|\xi|^{\alpha(x)}$, $x,\xi\in\bbr$, where $\alpha:\bbr \to [0,2]$ is Lipschitz continuous and satisfies $0<\underline\alpha = \inf\alpha(x)\leq \sup\alpha(x)\leq \overline\alpha<2$, corresponds to the so-called \emph{stable-like} Feller process, cf.\ \cite{bas88}. In this case
$\beta_\infty^x = \alpha(x)$ and $\beta_0 = \underline\alpha$.
\end{example}

\begin{lemma}\label{cndfzero}
    The only continuous negative definite function vanishing at infinity is constantly zero.
\end{lemma}

\begin{proof}
Let $\psi$ be a continuous negative definite function which vanishes at infinity. For every $\varepsilon >0$ there exists some  $R>0$ such that $|\psi(\xi)|\leq \varepsilon^2/4$ if $|\xi|>R$. For every $\gamma \in B_R (0)$ there exist two vectors $\xi,\eta \in B_R (0)^c$ such that $\gamma=\xi + \eta$. By the sub-additivity of $\sqrt{\abs{\psi}}$ we obtain
$$
    \sqrt{\abs{\psi(\gamma)}}
    =\sqrt{\abs{\psi(\xi+\eta)}}
    \leq \sqrt{\abs{\psi(\xi)}}+\sqrt{\abs{\psi(\eta)}}
    \leq \varepsilon
$$
which completes the proof.
\end{proof}

We can now simplify the calculation of the upper index.  The
assumptions of the following proposition are trivially satisfied by
any Feller process satisfying condition \eqref{rich}, cf.\
\cite{courrege}.

\begin{proposition}\label{prop:index}
    Let $p(x,\xi)$ be a non-trivial (i.e.\ non-constant) symbol of a Markov process  which is locally bounded. The generalized upper index $\beta_\infty^x$ can be calculated in the following way
    $$
    \beta_\infty^x
    =\beta(x)
    := \limsup_{\abs{\eta}\to\infty} \sup_{\abs{y-x}\leq 2/\abs{\eta}} \frac{\log\abs{p(y,\eta)}}{\log\abs{\eta}}.
$$
\end{proposition}
\begin{proof} First we show that $\beta(x)\in [0,2]$. Fix $x\in\bbr^d$. For $\abs{\eta} > 1$ we have only to consider points $y$ such that $\abs{y-x}\leq 2$. The argument used in the proof of Lemma \ref{lem:asymp} can be modified to prove that
$$
    |p(y,\eta)| \leq h(y)\cdot (1+|\eta|^2),
    \qquad
    h(y) = 4\sup_{|\xi|\leq 1} |p(y,\xi)|.
$$
Since $p(y,\eta)$ is locally bounded, we see there exists a constant
$C>0$ such that
$$
    \frac{\log\abs{p(y,\eta)}}{\log\abs{\eta}} \leq \frac{\log (2C) + \log\abs{\eta}^2}{\log\abs{\eta}}
    \leq \frac{\log (2C)}{\log{\abs{\eta}}} + 2.
$$
The right-hand side tends to $2$ as $\abs{\eta}\to\infty$. This shows that $\beta(x)\leq 2$. In order to see that $\beta(x)\geq 0$, we note that
$$
    \sup_{\abs{y-x}
    \leq 2/\abs{\eta}} \frac{\log\abs{p(y,\eta)}}{\log\abs{\eta}} \geq \frac{\log\abs{p(x,\eta)}}{\log\abs{\eta}}.
$$
Because of Lemma \ref{cndfzero} there exists a $\delta > 0$ such that for every $R>0$ there is some $\xi$ with $\abs{\xi}\geq R$ and $\abs{p(x,\xi)} > \delta$. Therefore,
$$
    \limsup_{\abs{\eta}\to\infty} \frac{\log\abs{p(x,\eta)}}{\log\abs{\eta}}
    \geq \limsup_{\abs{\eta}\to\infty}\frac{\log \delta}{\log\abs{\eta}} = 0,
$$
and we conclude that $\beta(x)\geq 0$.

In view of Lemma \ref{lem:asymp}, $\beta^x_\infty = \beta(x)$ follows, if we can show that
$$
    \limsup_{\abs{\xi}\to\infty}\frac{\sup_{\abs{x-y}\leq 2/ \abs{\xi}} \abs{p(y,\xi)}}{\abs{\xi}^\lambda}
    = 0\text{\ \ or\ \ }\infty
$$
according to $\lambda > \beta(x)$ or $\lambda <\beta(x)$. Let $h \in \bbr$. Then
\begin{align*}
    \frac{\sup_{\abs{x-y}\leq 2/ \abs{\xi}} \abs{p(y,\xi)}}{\abs{\xi}^{\beta(x)+h}}
    &= \exp \left( \log \left(\sup_{\abs{x-y}\leq 2/\abs{\xi}} \abs{p(y,\xi)} \right) - \left( \beta(x) +h \right)  \log{\abs{\xi}}\right) \\
    &=
    \exp \left(\left(\frac{\sup_{\abs{y-x}\leq 2/\abs{\xi}} \log\abs{p(y,\xi)}}{\log\abs{\xi}} - \beta(x)\right)\cdot\log{\abs{\xi}}-h \cdot\log{\abs{\xi}} \right).
\end{align*}
Taking the $\limsup$ for $\abs{\xi}\to \infty$ of this expression, the inner bracket converges to zero since $\beta(x) \in [0,2]$ as we have seen above. This means there exists some $r = r_h>0$ such that for every $R\geq r$
$$
    \left(\sup_{\abs{\xi}\geq R}\frac{\sup_{\abs{y-x}\leq 2/\abs{\xi}} \log\abs{p(y,\xi)}}{\log\abs{\xi}} - \beta(x)\right)<\frac{h}{2}.
$$
Thus, if $h>0$,
$$
    \limsup_{\abs{\xi}\to\infty} \frac{\sup_{\abs{x-y}\leq 2/ \abs{\xi}} \abs{p(y,\xi)}}{\abs{\xi}^{\beta(x)+h}} \leq \limsup_{\abs{\xi}\to\infty} \exp (\log (\abs{\xi}^{-h/2})) = 0;
$$
if $h <0$,
$$
    \limsup_{\abs{\xi}\to\infty} \frac{\sup_{\abs{x-y}\leq 2/ \abs{\xi}} \abs{p(y,\xi)}}{\abs{\xi}^{\beta(x)+h}}
    \geq \limsup_{\abs{\xi}\to\infty} \exp (\log (\abs{\xi}^{-h/2})) = \infty,
$$
which completes the proof.
\end{proof}

\begin{theorem}\label{theorem:sdeindex}
Let $X$ be a solution process of the SDE \eqref{sde}  with $d=n$ and
 where the linear mapping $\xi\mapsto\Phi^\top\!(y)\xi$ is
 bijective  for every $y\in\bbr^d$. If the driving
L\'evy process has the non-constant symbol $\psi$ and index
$\beta_\infty^\psi$, then the solution $X$ of the SDE has, for every
$x\in\bbr^d$, the upper index
$\beta_\infty^x\equiv\beta_\infty^\psi$.
\end{theorem}

\begin{proof}
Fix $x\in\bbr^d$. We use the characterization of the index from
Proposition \ref{prop:index}
$$
\beta_\infty^x= \limsup_{\abs{\eta}\to\infty} \sup_{\abs{y-x}\leq 2/\abs{\eta}} \frac{\log\abs{p(y,\eta)}}{\log\abs{\eta}}.
$$
From Theorem \ref{sde-symbol} we know that
$p(x,\xi)=\psi(\Phi^\top\!(x)\xi)$. Therefore,
$$
    \frac{\log \abs{\psi ( \Phi^\top\!(y)\eta )}}{\log \abs{\eta}}
    = \frac{\log \abs{\psi ( \Phi^\top\!(y)\eta )}}{\log \abs{\Phi^\top\!(x)\eta}} \cdot \frac{\log \abs{\Phi^\top\!(x)\eta}}{\log \abs{\eta}}
$$
where the second factor is bounded from above and below, since the
function $\eta \mapsto \phi^\top\!(x)\eta$ is  bijective.
Consequently,
$$
    \beta_\infty^x
    = \limsup_{\abs{\eta}\to\infty} \sup_{\abs{y-x}\leq 2/\abs{\eta}} \frac{\log \abs{\psi (\Phi^\top\!(y)\eta )}}{\log \abs{\Phi^\top\!(x)\eta}}.
$$
By Lemma \ref{cndfzero} $\psi$ does not vanish at infinity. In
particular there exists an $\varepsilon>0$ and a sequence
$(\xi_n)_{n\in \bbn}$ such that $\abs{\xi_n}\to\infty$ and
$\abs{\psi(\xi_n)}>\varepsilon$ for every $n\in\bbn$. Since
$\eta\mapsto\Phi^\top\!(x)\eta$ is linear and bijective there exists
a sequence $(\eta_n)_{n\in\bbn}$ such that
$\abs{\psi(\Phi^\top\!(x)\eta_n)}>\varepsilon$ for every $n\in\bbn$
and $\abs{\eta_n}\to\infty$.  In order to calculate the upper limit
it is, therefore, enough to consider the set
\begin{align} \label{subset}
    \big\{\eta \in \bbr^d : \abs{\psi(\Phi^\top\!(x)\eta)}\geq \varepsilon\big\}.
\end{align}

We write
\begin{equation}\label{zeroadd}\begin{aligned}
  \sup_{\abs{y-x}\leq 2/\abs{\eta}} & \frac{\log \abs{\psi ( \Phi^\top\!(y)\eta )}}{\log \abs{\Phi^\top\!(x)\eta}}\\
  &=\frac{\sup_{\abs{y-x}\leq 2/\abs{\eta}}\log \abs{\psi ( \Phi^\top\!(y)\eta )}-\log \abs{\psi(\Phi^\top\!(x)\eta)}}{\log \abs{\Phi^\top\!(x)\eta}}
   + \frac{\log \abs{\psi(\Phi^\top\!(x)\eta)}}{\log \abs{\Phi^\top\!(x)\eta}}.
\end{aligned}\end{equation}
Denoting the local Lipschitz constant of $y\mapsto \Phi(y)$ in a neighbourhood of $x$ by $L_x \geq 0$, we obtain for $\abs{y-x}\leq 2/\abs{\eta}$
\begin{align*}
\abs{\Phi^\top\!(y)\eta - \Phi^\top\!(x)\eta } \leq \abs{\eta} \cdot \abs{\Phi(y)-\Phi(x)} \leq \abs{\eta} \cdot L_x\abs{y-x}
\leq  2\,L_x.
\end{align*}
It follows that the numerator of the first term on the right-hand
side of \eqref{zeroadd} is bounded  on the set \eqref{subset}
because the function $y\mapsto \log \abs{y}$ is uniformly continuous
on $[\varepsilon,\infty[$;  for the second term we obtain
$$
    \limsup_{\abs{\eta}\to\infty}  \frac{\log \abs{\psi(\Phi^\top\!(x)\eta)}}{\log \abs{\Phi^\top\!(x)\eta}}
    = \limsup_{\abs{\xi}\to\infty} \frac{\log\abs{\psi(\xi)}}{\log\abs{\xi}}
    =\beta_\infty^\psi
$$
since the function $\eta \mapsto \phi^\top\!(x)\eta$ is bijective and linear.
\end{proof}

\begin{remark}\label{rem:index}
In order to obtain $\beta_\infty^x \leq \beta_\infty^\psi$ in the
case $d\leq n$  it is sufficient to demand that $\Phi(y)$ never
vanishes.
\end{remark}

We will first use this theorem to derive a result on the (strong)
$\gamma$-variation of the process $X$.
\begin{definition}\label{def:gammavariation}
    If $\gamma\in]0,\infty[$ and $g$ is an $\bbr^d$-valued function on the interval $[a,b]$ then
    $$
        V^\gamma(g; [a,b]) := \sup_{\pi_n} \sum_{j=1}^n
        \abs{g(t_j)-g(t_{j-1})}^\gamma
    $$
    where the supremum is taken over all partitions $\pi_n = (a=t_0 <t_1 < \ldots < t_n =b)$ of $[a,b]$ is called the  \emph{(strong) $\gamma$-variation} of $g$ on $[a,b]$.
\end{definition}

\begin{corollary}\label{cor:variation}
    Let $X^x=(X^x_t)_{t\geq 0}$ be the solution of the SDE \eqref{sde} where $Z$ is a L\'evy process with characteristic exponent $\psi$. Denote by $\beta_\infty^x$ the generalized upper index of $X$. Then
    $$
        V^\gamma(X^x; [0,T]) < \infty\quad \Pp^x\text{-a.s.\ for every\ \ } T>0
    $$
    if $\gamma>\sup_x \beta_\infty^x$. \newline
In the situation of Theorem \ref{theorem:sdeindex} the index $\sup_x
\beta_\infty^x = \beta_\infty^\psi$ where $\beta_\infty^\psi$ is the
upper index of the driving L\'evy process.
\end{corollary}
\begin{proof}[Proof of Corollary \ref{cor:variation}]
    Since $X$ is a strong Markov process we can use a criterion for the finiteness of $\gamma$-variations due to Manstavi\v{c}ius \cite{manstavicius}. Consider for $h\in [0,T]$ and $r>0$
\begin{align*}
    \alpha(h,r)
    &=\sup \big\{ \Pp^x(\abs{X_t-x}\geq r) : x\in \bbr^d, 0\leq t \leq (h \wedge T) \big\} \\
    &\leq \sup_{t\leq h} \sup_{x\in\bbr^d} \Pp^x\left(\sup_{0\leq s\leq t} \abs{X_s-x}\geq
    r\right).
\end{align*}
Using Lemma 4.1 and Lemma 5.1 in \cite{schilling98} we obtain
$$
    \Pp^x\left(\sup_{0\leq s\leq t} \abs{X_s-x}\geq r\right)
    \leq C\cdot t \sup_{\abs{y-x}\leq 2r} \sup_{\abs{{\mathsf{e}}} \leq 1} \abs{p\left(y,\frac{{\mathsf{e}}}{r}\right)}
$$
where $C\geq 0$ is independent of $x$ and $t$. Hence,
\begin{align*}
    \alpha(h,r)
    &\leq \sup_{t\leq h} \sup_{x\in\bbr^d}  C \cdot t   \sup_{\abs{y-x}\leq 2r} \sup_{\abs{{\mathsf{e}}} \leq 1} \abs{p\left(y,\frac{{\mathsf{e}}}{r}\right)} \\
    &\leq  C \cdot h \sup_{x\in\bbr^d} \left(\sup_{\abs{y-x}\leq 2r} \sup_{\abs{{\mathsf{e}}} \leq 1} \abs{p\left(y,\frac{{\mathsf{e}}}{r}\right)}\right) \\
    & \leq   C \cdot h \sup_{x\in\bbr^d} \left( \sup_{\abs{\eta} \leq (1/r)} \sup_{\abs{y-x}\leq (2/ \abs{\eta})} \abs{p\left(y,\eta\right)}\right).
\end{align*}
From Lemma \ref{lem:asymp} we know that for every $\lambda > \sup_x
\beta_\infty^x$
$$
    \lim_{|\eta|\to \infty}\frac{\sup_{\abs{y-x}\leq (2/ \abs{\eta})} \abs{p\left(y,\eta\right)}}{\abs{\eta}^\lambda} =
    0.
$$
Therefore we find for every $x$ a compact set $K$ such that
$$
\sup_{\abs{y-x}\leq 2/ \abs{\eta}} \abs{p\left(y,\eta\right)} \leq \tilde{C}\cdot \abs{\eta}^\lambda \leq
\tilde{C}\cdot r^{-\lambda}
$$
on the complement of $K$. Since the right-hand side is independent of $x$, there exists an $r_0 >0$ such that for all $r\in ]0,r_0]$ we have
$$
\alpha(h,r)\leq C\cdot\tilde{C}\cdot \frac{h^1}{r^\lambda}
$$
which means that $X^{\sigma}$ ($\sigma=\sigma^0_R$) is in the class
$\cm(1, \sup_x \beta_\infty^x )$ of Manstavi\v{c}ius. The result
follows from \cite{manstavicius} Theorem 1.3, as $R\uparrow\infty$.
\end{proof}

We can use the indices to obtain information on the H\"{o}lder and growth behaviour of the solution of the SDE \eqref{sde}. As usual, we write $(X_\cdot -x)_t^*:= \sup_{0\leq s\leq t} \abs{X_s-x}$.

\begin{corollary}\label{cor:hoelder}
    Let $X^x=(X^x_t)_{t\geq 0}$ be the solution of the SDE \eqref{sde} where $Z$ is a L\'evy process with characteristic exponent $\psi$ satisfying the sector condition $|\Im\psi(\xi)| \leq c_0\Re\psi(\xi)$ for some constant $c_0>0$. Denote by $\beta_\infty^x$ and $\beta_0$ the generalized upper indices of $X$. Then
    \begin{gather*}
    \lim_{t\to 0} t^{- 1/\lambda} (X_\cdot -x)_t^* = 0 \text{\ \ if\ \ } \lambda > \sup_x\beta_\infty^x  \quad\text{and}\quad
    \lim_{t\to \infty} t^{- 1/\lambda} (X_\cdot -x)_t^* = 0 \text{\ \ if\ \ } 0 < \lambda < \beta_0.
    \end{gather*}
    Under the assumptions of Theorem \ref{theorem:sdeindex}, $\sup_x\beta_\infty^x$ is the upper index of the driving L\'evy process: $\beta_\infty^\psi$.
\end{corollary}
\begin{proof}
    This is a combination of Proposition \ref{prop:index} and Theorem \ref{theorem:sdeindex} with the abstract result from \cite{schilling98}, Theorems 4.3 and 4.6. For the growth result as $t\to\infty$ we need the sector condition for the symbol $p(x,\xi)$ which is directly inherited from the sector condition of $\psi$. Note that one can identify values for $\lambda$ by using, in general different, indices such that the above limits become $+\infty$.
\end{proof}

Let us finally indicate how we can measure the `smoothness' of the
sample paths of the solution of the SDE \eqref{sde}. Since we deal
with c\`adl\`ag-functions, in general, the right scale of function
spaces are (polynomially weighted) Besov spaces
$B_q^s(L^p((1+t^2)^{-\mu/2}\,dt))$ with parameters $p,q\in
(0,\infty]$ and $s,\mu>0$. We refer to the monographs by Triebel
\cite{tri-mono} and the survey \cite{devore} by DeVore for details.
Note that information on Besov regularity is important if one is
interested in the effectiveness of numerical adaptive algorithms for
the solutions of an SDE. In a deterministic context this is
discussed in \cite{devore}.

\begin{corollary}\label{cor:besov}
    Let $X^x=(X^x_t)_{t\geq 0}$ be the solution of the SDE \eqref{sde} where $Z$ is a L\'evy process with non-degenerate (i.e.\ non-constant) characteristic exponent $\psi$. Denote by $\beta_\infty^x$ and $\beta_0$ the generalized upper indices of $X$. Then we have almost surely
    \begin{gather*}
        \{t\mapsto X_t^x\} \in B_{q}^s(L^p((1+t^2)^{-\mu/2}\,dt))
        \quad\text{if}\quad
        s\cdot\sup_y\{p,q,\beta_\infty^y\}<1
        \text{\ \ and\ \ }\mu > \frac 1{\beta_0}+\frac 1p.
    \end{gather*}
    In particular we get locally
    \begin{gather*}
        \{t\mapsto X_t^x\} \in B_{q}^{s,\textup{loc}}(L^p(dt))
        \quad\text{if}\quad
        s\cdot\sup_y\{p,q,\beta_\infty^y\}<1
    \intertext{and}
        \{t\mapsto X_t^x\} \not\in B_{q}^{s,\textup{loc}}(L^p(dt))
        \quad\text{if}\quad
        sp > 1.
    \end{gather*}
\end{corollary}
\begin{proof}
    This is a consequence of Theorems 4.2 and 6.5 in \cite{sch-besov}. Note that, although all statements are in terms of Feller processes, only the existence of a symbol of the underlying process is required. In \cite{sch-besov} we assume that the smoothness index $s$ satisfies the condition $s> (p^{-1}-1)^+$. This restriction can be easily overcome by using the imbedding $B_q^s(L^p)\hookrightarrow B_r^t(L^p)$ which holds for all $s>t$, all $p\in (0,\infty]$ and all $r,q\in (0,\infty]$, see \cite{tri-mono}, vol.\ III, Theorem 1.97.
\end{proof}

\section{Appendix}
For the readers' convenience we collect in this appendix some variations on standard estimates for symbols of Feller processes. They are based on methods from \cite{schilling98} and can also be found in \cite{schnurr-diss}. For the rest of the paper we use the notation $\nnorm{f}_K := \sup_{z\in K} |f(z)|$ where $|\cdot|$ can be a vector or matrix norm.

\begin{lemma}\label{l1}
    We have
$$
    \frac{|y|^2}{1+|y|^2} = \int \left(1-\cos(y^\top \rho)\right)\, g(\rho)\,d\rho,
    \qquad y\in\rd,
$$
    where
$
    g(\rho) = \frac 12 \int_0^\infty
            (2\pi\lambda)^{-d/2}\,e^{-|\rho|^2/2\lambda}\,e^{-\lambda/2}\,d\lambda
$
    is integrable and has absolute moments of arbitrary order.
\end{lemma}
\begin{proof}
    The Tonelli-Fubini Theorem and a change of variables show for $k\in\bbn_0$
\begin{align*}
    \int |\rho|^k\,g(\rho)\,d\rho
    &=
    \frac{1}{2} \int_0^\infty (2\pi\lambda)^{-d/2} \int |\rho|^k\,e^{-|\rho|^2/2\lambda}
        \,d\rho \; e^{-\lambda/2}\,d\lambda\\
    &=
    \frac{1}{2} \int_0^\infty (2\pi\lambda)^{-d/2} \int \lambda^{k/2} |\eta|^k
        \,e^{-|\eta|^2/2} \lambda^{d/2}\,d\eta \; e^{-\lambda/2}\,d\lambda\\
    &=
    \frac{1}{2} (2\pi)^{-d/2} \int |\eta|^k \,e^{-|\eta|^2/2}\,d\eta\
    \int_0^\infty \lambda^{(k+d)/2}\,e^{-\lambda/2}\,d\lambda ,
\end{align*}
i.e., $g$ has absolute moments of any order. Moreover, the elementary formula
$$
    e^{-\lambda |y|^2/2} = (2\pi\lambda)^{-d/2}\int
                            e^{-|\rho|^2/2\lambda} \, e^{iy^\top \rho}\, d\rho
$$
and Fubini's Theorem yield
\begin{align*}
    \frac{|y|^2}{1+|y|^2}
&=
    \frac{1}{2}\int_0^\infty \big( 1-e^{-\lambda |y|^2/2}\big)\,e^{-\lambda/2}\,d\lambda\\   &=
    \frac{1}{2}\int_0^\infty\int (2\pi\lambda)^{-d/2} \big(1-e^{iy^\top\rho}\big)
    e^{-|\rho|^2/2\lambda}\, e^{-\lambda/2}\,d\rho\,d\lambda\\
&=
    \int \big(1-e^{iy^\top\rho}\big)\,g(\rho)\,d\rho .
\end{align*}
The assertion follows since the left-hand side is real-valued.
\end{proof}

\begin{lemma}\label{l3}
Let $p(x,\xi)$ be a negative definite symbol of the form \eqref{lkfx} with L\'evy triplet $(\ell(x),Q(x),N(x,dy))$ and let $K\subset\rd$ be a compact set or $K=\rd$. Then the following assertions are equivalent.
\begin{enumerate}
    \item[\upshape (a)] $\displaystyle\nnorm{p(\cdot,\xi)}_K\leq
                c_p(1+|\xi|^2),\qquad\xi\in\rd$;
    \item[\upshape (b)] $\displaystyle\nnorm\ell_K + \nnorm Q_K +
                \left\|\int_{y\neq 0} \frac{|y|^2}{1+|y|^2}\,N(\cdot,dy) \right\|_K
                < \infty$;
    \item[\upshape (c)] $\displaystyle\sup_{|\xi|\leq 1}\sup_{x\in K}|p(x,\xi)| < \infty$.
\end{enumerate}
    If one, hence all, of the above conditions hold, there exists a
    constant $c>0$ such that
$$
    \nnorm\ell_K + \nnorm Q_K +
                \left\|\int_{y\neq 0} \frac{|y|^2}{1+|y|^2}\,N(\cdot,dy) \right\|_K
    \leq
    c\,\sup_{|\xi|\leq 1}\sup_{x\in K}|p(x,\xi)|.
$$
\end{lemma}
\begin{proof}
    (a)$\Rightarrow$(b). By Lemma \ref{l1} we have
\begin{align*}
    \int_{y\neq 0} \frac{|y|^2}{1+|y|^2}\,N(x,dy)
    &=
    \int_{y\neq 0}\int \big(1-\cos (\eta^\top y)\big)\,g(\eta)\,d\eta\,N(x,dy)\\
    &=
    \int \big(\Re p(x,\eta) - \eta^\top Q(x)\eta\big)\,g(\eta)\,d\eta\\
    &\leq
    \int \Re p(x,\eta)\,g(\eta)\,d\eta\\
    &\leq
    c_p\int \left(1+|\eta|^2\right)\,g(\eta)\,d\eta
\end{align*}
uniformly for all $x\in K$. Using Taylor's formula and Lemma \ref{l1} we find
\begin{align*}
    |\ell(x)^\top\xi|
&\leq
    \big|\Im p(x,\xi)\big| + \Im \int_{y\neq 0} \left|1-e^{i\xi^\top y}
        +\frac{i\xi^\top y}{1+|y|^2} \right| \,N(x,dy)\\
&\leq
    \left(c_p + c\int_{y\neq 0} \frac{|y|^2}{1+|y|^2}\,N(x,dy) \right)\,(1+|\xi|^2)\\
&\leq
    c_p\left(1 + c\int \left(1+|\eta|^2\right)\,g(\eta)\,d\eta \right)\,(1+|\xi|^2)
\end{align*}
uniformly in $x\in K$ and for all $\xi\in\rd$, so $\nnorm\ell_K < \infty$. Finally,
$$
    |\xi^\top Q(x)\xi|
    \leq
    \Re p(x,\xi)
    \leq
    |p(x,\xi)|
    \leq
    c_p(1+|\xi|^2)
$$
which shows that $\nnorm Q_K < \infty$.

\medskip\noindent
(b)$\Rightarrow$(c). Using the L\'evy-Khinchine representation for
$p(x,\xi)$ and Taylor's formula we find
$$
    |p(x,\xi)|
    \leq
    \nnorm\ell_K \,|\xi| + \nnorm Q_K\,|\xi|^2
    + 2\int \frac{|y|^2}{1+|y|^2}\,N(x,dy) \, (1+|\xi|^2)
$$
(we use the $\ell^2$-norm in $\rd$ and $\bbr^{d\times d}$) and (c) follows.

\medskip\noindent
(c)$\Rightarrow$(a). Set $P(\xi):=\sup_{x\in K}|p(x,\xi)|$. Since both $\xi\mapsto\sqrt{p(x,\xi)}$ and the $\sup_x$ are subadditive, we conclude
$$
    \sqrt{P(\xi+\eta)}
    \leq \sqrt{P(\xi)} + \sqrt{P(\eta)}, \qquad \xi,\eta\in\rd,
$$
i.e., $\sqrt{P(\cdot)}$ is subadditive. Fix $\xi$ and choose the unique $N = N_\xi\in\bbn$ such that $N-1\leq |\xi| < N$. Applying the subadditivity estimate $N$ times gives
$$
    P(\xi)
    \leq N^2P\left(\tfrac\xi N \right)
    \leq N^2 \sup_{|\eta|\leq 1} P(\eta)
    \leq 2\left(1+|\xi|^2\right)\sup_{|\eta|\leq 1} P(\eta)
$$
and this is claimed in (a).

An inspection of the proof of (a)$\Rightarrow$(b) shows that each of
the terms $\nnorm\ell_K$, $\nnorm Q_K$ and $\left\|\int_{y\neq 0}
|y|^2(1+|y|^2)^{-1}\,N(\cdot,dy) \right\|_K$ is bounded by constants
of the form $c\cdot c_p$ where $c_p$ is from (a). The proof of
(c)$\Rightarrow$(a) reveals that $c_p = 2\sup_{x\in
K}\sup_{|\eta|\leq 1}|p(x,\eta)|$ which proves the last statement.
\end{proof}



\begin{thebibliography}{99}
\setlength{\itemsep}{-.7mm}
\small
\bibitem{applebaum}
Applebaum, D.: \emph{L\'evy Processes and Stochastic Calculus}. Cambridge University Press, Cambridge 2009 (2nd ed).

\bibitem{bas88}
Bass, R.\,F.: Uniqueness in Law for Pure Jump Type Markov Processes. \emph{Probab.\ Theory Relat.\ Fields} \textbf{79} (1988), 271--287.

\bibitem{bergforst}
Berg, C. and Forst, G.: \emph{Potential Theory on Locally Compact Abelian Groups}, Springer-Verlag, Berlin 1975.

\bibitem{blumenthalget}
Blumenthal, R.~M. and Getoor, R.~K.: \emph{Markov Processes and Potential Theory}, Academic Press, New York 1968.

\bibitem{bot-sch09}
B\"{o}ttcher, B.\ and Schilling, R.\ L.: Approximation of Feller processes by Markov chains with L\'{e}vy increments. \emph{Stochastics and Dynamics} \textbf{9} (2009) 71--80.

\bibitem{vierleute}
Cinlar, E., Jacod, J., Protter, P., and Sharpe, M.~J.: Semimartingales and Markov processes. \emph{Z.\ Wahrscheinlichkeitstheorie verw.\ Geb.}, \textbf{54} (1980), 161--219.

\bibitem{courrege}
Courr\`{e}ge, P.: Sur la forme int\'egro-diff\'erentielle des op\'erateurs de ${C}_k^\infty$ dans ${C}$ satisfaisant au principe du maximum. \emph{S\'em.\ Th\'eorie du potentiel} (1965/66) Expos\'e 2, 38pp.

\bibitem{devore}
DeVore, R.: Nonlinear Approximation. \emph{Acta Numerica} \textbf{7} (1998), 51-150.

\bibitem{hoh98}
Hoh, W.: A symbolic calculus for pseudo differential operators generating Feller semigroups, \emph{Osaka J.\ Math.} \textbf{35} (1998), 798--820.

\bibitem{hoh00}
Hoh, W.: Pseudo differential operators with negative definite symbols of variable order, \emph{Rev.\ Mat.\ Iberoam.} \textbf{16} (2000), 219--241.

\bibitem{hoh02}
Hoh, W., On perturbations of pseudo differential operators with negative definite symbol, \emph{Appl.\ Anal.\ Optimization} \textbf{45} (2002), 269--281.

\bibitem{ikedawat}
Ikeda, N. and Watanabe, S.: \emph{Stochastic Differential Equations and Diffusion Processes}. North Holland, Amsterdam 1989 (2nd ed).

\bibitem{jac-sch-lkf}
Jacob, N. and Schilling, R.~L.: An analytic proof of the L\'evy-Khinchin formula on $\bbr^n$. \emph{Publ. Math. Debrecen} \textbf{53} (1998), 69--89.

\bibitem{jac-sch-survey}
Jacob, N. and Schilling, R.~L.: L\'{e}vy-type processes and pseudo differential operators. In: Barndorff-Nielsen, O.\ E.\ et al.\ (eds.): \emph{L\'evy Processes: Theory and Applications}, Birkh\"auser, Boston 2001, 139--168.

\bibitem{jacob1}
Jacob, N.: \emph{Pseudo Differential Operators and Markov Processes. Vol.\ 1: Fourier Analysis and Semigroups}, Imperial College Press, London 2001.

\bibitem{jacob2}
Jacob, N.: \emph{Pseudo Differential Operators and Markov Processes. Vol.\ 2: Generators and their Potential Theory}, Imperial College Press, London 2002.

\bibitem{jacob3}
Jacob, N.: \emph{Pseudo Differential Operators and Markov Processes. Vol.\ 3: Markov Processes and Applications}, Imperial College Press, London 2005.

\bibitem{jacodshir}
Jacod, J. and Shiryaev, A.: \emph{Limit Theorems for Stochastic Processes}. Springer-Verlag, Berlin 1987.

\bibitem{jacodprotter}
Jacod, J. and Protter, P.: Une remarque sur les \'equation diff\'erencielles stochastiques \`a solution markoviennes. In: \emph{S\'eminaire de Probabilit\`es XXV}, Lecture Notes in Mathematics vol.\ 1485, Berlin 1991, 138--139.

\bibitem{kas09}
Ka{\ss}mann, M.: A priori estimates for integro-differential operators with measurable kernels. \emph{Calc.\ Var.\ Partial Differ.\ Equ.} \textbf{34} (2009), 1--21.

\bibitem{manstavicius}
Manstavi\v{c}ius, M.: p-variation of strong Markov processes. \emph{Ann.\ Probab.} \textbf{32} (2004), 2053--2066.

\bibitem{metivier}
M\'etivier, M.: \emph{Semimartingales. A Course on Stochastic Processes}. Walter de Gruyter, Berlin 1982.

\bibitem{protter77}
Protter, P.: Markov Solutions of stochastic differential equations. \emph{Z.\ Wahrscheinlichkeitstheorie verw.\ Geb.} \textbf{41} (1977), 39--58.

\bibitem{protter}
Protter, P.: \emph{Stochastic Integration and Differential Equations}. Springer-Verlag, Berlin 2005 (2nd ed).

\bibitem{revuzyor}
Revuz, D. and Yor, M.: \emph{Continuous Martingales and Brownian Motion}. Springer-Verlag, Berlin 1999 (3rd ed).

\bibitem{sato}
Sato, K.: \emph{L\'evy Processes and Infinitely Divisible Distributions}. Cambridge University Press, Cambridge 1999.

\bibitem{schilling98pos}
Schilling, R.~L.: Conservativeness and extensions of Feller semigroups. \emph{Positivity} \textbf{2} (1998), 239--256.

\bibitem{schilling98}
Schilling, R.~L.: Growth and H\"older conditions for the sample paths of Feller processes. \emph{Probab.\ Theory Rel.\ Fields} \textbf{112} (1998), 565--611.

\bibitem{sch-besov}
Schilling, R.\,L.: Function spaces as path spaces of Feller processes, \emph{Math.\ Nachr.} \textbf{217} (2000), 147--174.

\bibitem{schnurr-diss}
Schnurr, A.: \emph{The Symbol of a Markov Semimartingale}. PhD thesis, TU Dresden 2009.

\bibitem{tri-mono}
Triebel, H.: \emph{The Theory of Function Spaces I, II, III}. Birkh\"{a}user, Basel 1983--2006.
\end{thebibliography}
\end{document}